\documentclass[12pt]{article}
\usepackage{amsmath,amssymb,amsthm}
\usepackage{latexsym}
\usepackage{graphics, graphicx}
\usepackage{pstricks, psfrag}

\numberwithin{equation}{section}

\newcommand{\cH}{\ensuremath{\mathcal{H}}}

\newcommand{\cP}{\ensuremath{\mathcal{P}}}

\newtheorem{Thm}{Theorem}[section]
\newtheorem{Lem}[Thm]{Lemma}
\newtheorem{Prop}[Thm]{Proposition}

\newtheorem*{Ex}{Example}
\newtheorem*{Prob}{Problem}

\newtheorem{Clm}{Claim}

\makeatother
\title{Holes and a chordal cut in a graph}

\author{
\begin{tabular}{c}
{\sc Suh-Ryung KIM}
\thanks{This research was supported by Basic Science Research Program
through the National Research Foundation of Korea (NRF) funded
by the Ministry of Education, Science and Technology (700-20100058).}\\
[1ex]
Department of Mathematics Education \\
Seoul National University, Seoul 151-742, Korea\\
{\tt srkim@snu.ac.kr}\\
\\
{\sc Jung Yeun LEE}
\thanks{This work was supported by NAP of Korea Research Council
of Fundamental Science and Technology.}
\thanks{Corresponding author}\\
[1ex]
National Institute for Mathematical Sciences \\
Daejeon 305-390, Korea \\
{\tt jungyeunlee@gmail.com}\\
\\
{\sc Yoshio SANO}
\thanks{This work was supported by Priority Research Centers Program
through the National Research Foundation of Korea (NRF) funded by
the Ministry of Education, Science and Technology
(MEST) (No. 2010-0029638).}\\
[1ex]
Pohang Mathematics Institute \\
POSTECH, Pohang 790-784, Korea \\
{\tt ysano@postech.ac.kr}
\end{tabular} }

\date{}

\begin{document}

\maketitle

\newpage

\begin{abstract}
A set $X$ of vertices of a graph $G$ is called a {\em clique cut} of $G$
if the subgraph of $G$ induced by $X$ is a complete graph
and the number of connected components of $G-X$ is greater than that of $G$.
A clique cut $X$ of $G$ is called a {\em chordal cut} of $G$
if there exists a union $U$ of connected components
of $G-X$ such that $G[U \cup X]$ is a chordal graph.

In this paper, we consider the following problem:
Given a graph $G$, does the graph have a chordal cut?
We show that $K_{2,2,2}$-free hole-edge-disjoint graphs
have chordal cuts if they satisfy a certain condition.
\end{abstract}



\section{Introduction and Preliminaries}
\subsection{Introduction}

A set $X$ of vertices of a graph $G$ is called a {\em clique} of $G$
if the subgraph of $G$ induced by $X$ is a complete graph.
A set $X$ of vertices of a graph $G$ is called a {\em vertex cut} of $G$
if the number of connected components of $G-X$ is greater than that of $G$.
We call a vertex cut $X$ of $G$ a {\em clique cut} of $G$
if $X$ is a clique of $G$.
A clique cut $X$ of $G$ is called a {\em chordal cut} of $G$
if there exists a union $U$ of connected components
of $G-X$ such that $G[U \cup X]$ is a chordal graph.

In this paper, we consider the following problem:

\begin{Prob}
Given a graph $G$, does the graph have a chordal cut?
\end{Prob}

\begin{Ex}
{\rm
A cycle has no clique cut and so has no chordal cut.
}
\end{Ex}

\begin{Ex}
{\rm
Every clique cut of a chordal graph is a chordal cut.
}
\end{Ex}

This paper is organized as follows:
Subsection \ref{subsec:moti} describes
the authors' motivation to consider this problem.
Subsection \ref{subsec:prel} prepares some lemmas
which will be used in the following sections.
Section \ref{sec:main} presents a main result of this paper.
We show that $K_{2,2,2}$-free hole-edge-disjoint graphs
have chordal cuts if they satisfy a certain condition.
Section \ref{sec:comp} shows that
the main result can be applied to obtain
a sharp upper bound for the competition numbers of graphs.

\subsection{Motivation}\label{subsec:moti}

Some problems on the competition numbers of graphs
motivated us to consider the problem stated above.

The {\em competition graph} of a digraph $D$, denoted by $C(D)$,
has the same set of
vertices as $D$ and an edge between vertices $u$ and $v$ if and only if
there is a vertex $x$ in $D$ such that $(u,x)$ and $(v,x)$ are arcs of
$D$.
The notion of competition graph was introduced by Cohen~\cite{co} as a
means of determining the smallest dimension of ecological phase space.
Roberts~\cite{cn} observed that any graph together
with sufficiently many isolated vertices is the competition graph of an
acyclic digraph.  Then he defined the {\em competition number $k(G)$} of
a graph $G$ to be the smallest number $k$ such that $G$ together with
$k$ isolated vertices added is the competition graph of an acyclic
digraph.
It does not seem to be easy in general to compute $k(G)$ for all graphs
$G$, as Opsut~\cite{op} showed that the computation of the
competition number of a graph is an NP-hard problem.
It has been one of important research problems in the study of
competition graphs to characterize a graph by its competition
number.

Now we recall a theorem which shows that a chordal cut of a graph
is related to its competition number.

\begin{Thm}[{\cite[Theorem 2.2]{twoholes}}]\label{chordalpart}
Let $G$ be a graph and $k$ be a nonnegative integer.
Suppose that $G$ has a subgraph $G_1$ with $k(G_1) \leq k$
and a chordal subgraph $G_2$ such that
$E(G_1) \cup E(G_2) =E(G)$
and $V(G_1) \cap V(G_2)$ is a clique of $G_2$.
Then $k(G) \le k+1$.
\end{Thm}

\noindent
In this theorem, since $E(G_1) \cup E(G_2)=E(G)$, 
$V(G_1)\cap V(G_2)$ is a clique cut of $G$
if $V(G_2) \setminus V(G_1) \neq \emptyset$ 
and $V(G_1) \setminus V(G_2) \neq \emptyset$. 
Moreover, $V(G_2)\setminus V(G_1)$ and $V(G_1) \cap V(G_2)$ 
induce the chordal graph $G_2$. 
Thus $V(G_1) \cap V(G_2)$ is a chordal cut of $G$ if 
$V(G_2) \setminus V(G_1) \neq \emptyset$ 
and $V(G_1) \setminus V(G_2) \neq \emptyset$.

\subsection{Preliminaries}\label{subsec:prel}

First let us fix basic terminology.
A {\it walk} in a graph $G$ is a vertex sequence
$v_0 v_1 \cdots v_{l-1} v_l$ such that $v_iv_{i+1}$ is an edge of $G$
for $0 \leq i \leq l-1$.
The number $l$ is called the {\it length} of the walk.
We refer to the vertices $v_0$ and $v_l$ as the {\it end vertices}
and the vertices $v_1, \ldots, v_{l-1}$ as the {\it internal vertices}
of the walk.
A {\it $(u,v)$-walk} is a walk $v_0 v_1 \cdots v_{l-1} v_l$ with
$v_0=u$ and $v_l=v$.
For a walk $W$ in $G$, we denote by $W^{-1}$ the
walk represented by the reverse of the vertex sequence of $W$.
A walk $v_0 v_1 \cdots v_{l-1} v_l$ is called a {\it path}
if all the vertices $v_0, v_1, \ldots, v_{l}$ of the walk are distinct.
A {\it $(u,v)$-path} is a $(u,v)$-walk which is a path.
A walk $v_0 v_1 \cdots v_{l-1} v_0$ is called a {\it cycle}
if $v_0 v_1 \cdots v_{l-1}$ is a path, where $l \geq 3$.
The indices of cycles are considered in modulo the length of the cycle.
A {\it chord} for a cycle $v_0 v_1 \cdots v_{l-1} v_0$ is an edge
$v_i v_j$ with $|i-j| \geq 2$.
A {\it chordless cycle} is a cycle having no chord.
A {\it hole} is a chordless cycle of length at least $4$.
We denote
the set of holes in a graph $G$ by $\cH(G)$ and
the number of holes in a graph $G$ by $h(G)$.
A cycle of length $3$ is called a {\it triangle}.

For a hole $C$ in a graph $G$,
we denote by $X_C^{(G)}$ the set of vertices which are adjacent
to all the vertices of $C$:
\begin{equation}\label{eq:1}
X_C^{(G)} = \{v \in V(G) \mid uv \in E(G) \text{ for all } u \in V(C) \}.
\end{equation}
For a graph $G$ and a hole $C$ of $G$,
we call a walk (resp.\ path) $W$
a {\em $C$-avoiding walk} (resp.\ {\em $C$-avoiding path})
if the following hold:
\begin{itemize}
\item[(1)]
None of the internal vertices of $W$ are in $V(C) \cup X_C^{(G)}$,
\item[(2)]
If the length of $W$ is $1$, then
one of the two vertices of $W$ is not in $V(C) \cup X_C^{(G)}$.
\end{itemize}
Let $\cP_{C,uv}^{(G)}$ denote the set of
all $C$-avoiding $(u,v)$-paths in $G$.
For a hole $C \in \cH(G)$ of a graph $G$ and an edge $e=uv \in E(C)$
of the hole $C$,
we define
\begin{eqnarray}
\label{eq:2}
X_{C,e}^{(G)} = X_{C,uv}^{(G)} &:=& X_C^{(G)} \cup \{u,v\}, \\
\label{eq:3}
S_{C,e}^{(G)} = S_{C,uv}^{(G)} &:=& \bigcup_{P \in \cP_{C,uv}^{(G)}} V(P)
\setminus \{u,v\}, \\
\label{eq:4}
T_{C,e}^{(G)} = T_{C,uv}^{(G)} &:=&
\{ w \in V(G) \mid uwv
\in \cP_{C,uv}^{(G)} \}.
\end{eqnarray}
It is easy to see that $T_{C,e}^{(G)} \subseteq S_{C,e}^{(G)}$ and
$S_{C,e}^{(G)} \cap X_{C,e}^{(G)} = \emptyset$.
Also note that the set $S_{C,uv}^{(G)}$ is not empty
if and only if $G$ has a $C$-avoiding $(u,v)$-path.

We call a graph $G$ a {\em hole-edge-disjoint graph} if all the
holes of $G$ are mutually edge-disjoint.
We say that a graph is {\em $K_{2,2,2}$-free} if it does not
contain the complete tripartite graph $K_{2,2,2}$ as an induced subgraph.
The following are some fundamental properties
of $K_{2,2,2}$-free hole-edge-disjoint graphs.

\begin{Lem}\label{lem:1}
Let $G$ be a $K_{2,2,2}$-free hole-edge-disjoint graph
and let $C \in \cH(G)$.
Then the following hold: \\
{\rm (1)}
$G$ has no $C$-avoiding path between two non-adjacent vertices of $C$, \\
{\rm (2)}
$X_C^{(G)}$ is a clique.
\end{Lem}

\begin{proof}
It follows from \cite[Theorem 2.18]{E1}.
\end{proof}

\begin{Lem}\label{lem:2}
Let $G$ be a $K_{2,2,2}$-free hole-edge-disjoint graph
and let $C \in \cH(G)$ and $e \in E(C)$.
If $S_{C,e}^{(G)} \neq \emptyset$,
then $X_{C,e}^{(G)}$ is a vertex cut of $G$.
\end{Lem}

\begin{proof}
If the length of $C$ is at least $5$,
then it follows from \cite[Lemma 2.11]{E1}.
We can also prove the lemma similarly when the length of $C$ is equal to $4$
since  $G$ is $K_{2,2,2}$-free.
\end{proof}

\begin{Prop}\label{prop:12}
Let $G$ be a $K_{2,2,2}$-free hole-edge-disjoint graph
and let $C\ \in \cH(G)$ and $e \in E(C)$.
If $S_{C,e}^{(G)} \neq \emptyset$,
then $X_{C,e}^{(G)}$ is a clique cut of $G$.
\end{Prop}

\begin{proof}
It follows from Lemmas \ref{lem:1} and \ref{lem:2}.
\end{proof}

\begin{Lem}\label{lem:3}
Let $G$ be a $K_{2,2,2}$-free hole-edge-disjoint graph
and let $C \in \cH(G)$ and $e \in E(C)$.
If $S_{C,e}^{(G)} = \emptyset$,
then
the graph $G-e$ obtained from $G$ by deleting the edge $e$
is a $K_{2,2,2}$-free hole-edge-disjoint graph
with $h(G-e) \leq h(G)-1$.
\end{Lem}

\begin{proof}
It follows from \cite[Lemma 3.1]{E1}.
\end{proof}

We close this subsection with noting that the set $S_{C,e}^{(G)}$ in the statements
in Lemma \ref{lem:2}, Proposition \ref{prop:12}, and Lemma \ref{lem:3}
(also in Theorems \ref{thm:chordalcut} and \ref{thm:1})
can be replaced by the set $T_{C,e}^{(G)}$.

\begin{Prop} \label{ST}
Let $G$ be a hole-edge-disjoint graph
and let $C \in \cH(G)$ and $e \in E(C)$.
Then,
$S_{C,e}^{(G)} \neq \emptyset$ if and only if
$T_{C,e}^{(G)} \neq \emptyset$.
\end{Prop}

\begin{proof}
Since $T_{C,e}^{(G)} \subseteq S_{C,e}^{(G)}$,
$T_{C,e}^{(G)} \neq \emptyset$ implies
$S_{C,e}^{(G)} \neq \emptyset$.
Now we show the ``only if" part.
Let $e=uv$ and
suppose that $S_{C,uv}^{(G)} \neq \emptyset$.
Then $G$ has a $C$-avoiding $(u,v)$-path.
Let $P :=u w_1 \cdots w_{l-1} v$ be a shortest path
among all $C$-avoiding $(u,v)$-path in $G$,
where $l \geq 2$.
Then the path $P$ and the edge $uv$ form an induced cycle $C'$
which share the edge $uv$ with the hole $C$.
Since $G$ is hole-edge-disjoint, $C'$ must be a triangle.
Therefore, the length $l$ of $P$ is equal to $2$, i.e.,
$P=uw_1v$.
Hence $w_1 \in T_{C,uv}^{(G)}$,
that is, $T_{C,uv}^{(G)} \neq \emptyset$.
\end{proof}

\section{Main Results}\label{sec:main}

For a hole $C$ in a graph $G$ and an edge $e$ of $C$,
if there is no confusion,
we denote the sets $X_C^{(G)}$, $X_{C,e}^{(G)}$, $S_{C,e}^{(G)}$,
and $T_{C,e}^{(G)}$ defined by (\ref{eq:1})-(\ref{eq:4})
simply by $X_C$, $X_{C,e}$, $S_{C,e}$, and $T_{C,e}$, respectively.

Let $G$ be a $K_{2,2,2}$-free hole-edge-disjoint graph.
For a hole $C$ in $G$ and an edge $e=uv$ of $C$, 
we define $Q_{C,e}$ and $U_{C,e}$ as follows 
(see Figure \ref{fig-ex} for an illustration):

\begin{center}
\begin{tabular}{|c|p{4.3in}|}
\hline
$Q_{C,e}$ &
the connected component of $G-X_{C,e}$ containing
$V(C)\setminus\{u,v\}$ \\
\hline
$U_{C,e}$ &
the union of the connected components of $G-X_{C,e}-Q_{C,e}$
each of which contains a vertex in $T_{C,e}$ \\
\hline
\end{tabular}
\end{center}
Note that $T_{C,e} \subseteq V(U_{C,e})$. 
We say that $G$ has the {\em chordal property}
if there exist $C \in \cH(G)$ and $e\in E(C)$ such that
the graph
$G[V(U_{C,e}) \cup X_{C,e}]$ is chordal.
In this definition, $X_{C,e}$ is a chordal cut of $G$ 
if $T_{C,e} \neq \emptyset$. 

\begin{figure}[!hp]
\psfrag{G}{$G$}
\psfrag{Gce}{$U_{C,e}$}
\psfrag{Qce}{$Q_{C,e}$}
\psfrag{X}{$X_{C,e}$}
\psfrag{G-X}{$G-X_{C,e}$}
\begin{center}
\includegraphics[scale=0.8]{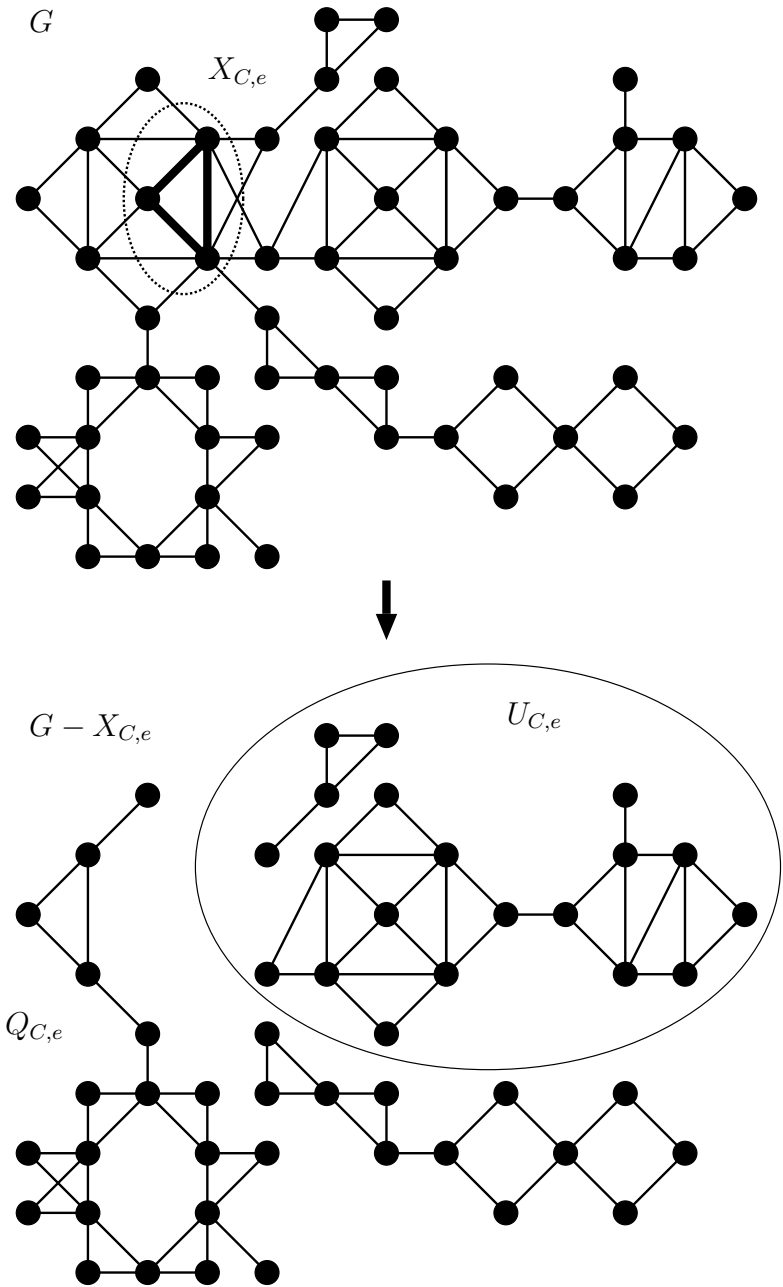}
\caption{$Q_{C,e}$ and $U_{C,e}$ in $G-X_{C,e}$}
\label{fig-ex}
\end{center}
\end{figure}

Now we present our main result:

\begin{Thm}\label{thm:chordalcut}
Let $G$ be a $K_{2,2,2}$-free hole-edge-disjoint graph.
Suppose that $S_{C,e}^{(G)} \neq \emptyset$
for any $C \in \cH(G)$ and any $e \in E(C)$.
Then $G$ has the chordal property.
Consequently, $G$ has a chordal cut.
\end{Thm}

\begin{proof} 
Since $S_{C,e} \neq \emptyset$ 
for any $C \in \cH(G)$ and any $e \in E(C)$, 
$T_{C,e} \neq \emptyset$ holds for any $C \in \cH(G)$ and any $e \in E(C)$ 
by Proposition~\ref{ST}. 
In addition, 
by Proposition \ref{prop:12}, 
$X_{C,e}$ is a clique cut of $G$
for any $C \in \cH(G)$ and any $e \in E(C)$. 
We will show that one of 
these clique cuts is a chordal cut of $G$  
by contradiction. 

Suppose that $G$ does not have the chordal property. 
Then, for any $C \in \cH(G)$ and any $e \in E(C)$, the graph 
$G[V(U_{C,e}) \cup X_{C,e} ]$ contains a hole. 
Now we fix a hole $C^* \in \cH(G)$ and an edge $e^* \in E(C^*)$. 
By our assumption, 
$G[V(U_{C^*,e^*}) \cup X_{C^*,e^*}]$ contains a hole. 
Then, for any hole $C$ in $G[V(U_{C^*,e^*}) \cup X_{C^*,e^*}]$, 
the following Claims 1-4 hold:

\begin{Clm}\label{Clm-1}
If there exists $e=uv \in E(C)$ such that
$T_{C,e} \not\subseteq V(U_{C^*,e^*}) \cup X_{C^*,e^*}$,
then $\{u,v \} \subset X_{C^*,e^*}$
\end{Clm}

\begin{proof}[Proof of Claim \ref{Clm-1}]
Since $T_{C,e} \not\subseteq V(U_{C^*,e^*}) \cup X_{C^*,e^*}$,
there exists a vertex $w$ in
$T_{C,e} \setminus (V(U_{C^*,e^*}) \cup X_{C^*,e^*})$.
Suppose that
one of $u$, $v$ is not contained in $X_{C^*,e^*}$.
Without loss of generality, we may assume that $u$
is not in $X_{C^*,e^*}$.
Then
$u$ is contained in $V(U_{C^*,e^*})$
since $u \in V(C)$ and
$V(C)$ is contained in $V(U_{C^*,e^*}) \cup X_{C^*,e^*}$.
Since $w \not\in X_{C^*,e^*}$, the vertices 
$u$ and $w$ are still adjacent in $G-X_{C^*,e^*}$.
However $w \not\in V(U_{C^*,e^*})$ 
while $u \in V(U_{C^*,e^*})$.
This implies that $u$ and $w$ belong to different components
of $G- X_{C^*,e^*}$ and so we reach a contradiction.
Thus, $\{u,v \} \subset X_{C^*,e^*}$.
\end{proof}

\begin{Clm}\label{Clm-2}
$|\{e \in E(C) \mid T_{C,e} \subset V(U_{C^*,e^*}) \}| \geq |E(C)|-2$,
\end{Clm}

\begin{proof}[Proof of Claim \ref{Clm-2}]
By contradiction.
Suppose that
$|\{e \in E(C) \mid T_{C,e} \subset V(U_{C^*,e^*}) \}| < |E(C)|-2$.
Then
$|\{e \in E(C) \mid T_{C,e} \not\subset V(U_{C^*,e^*}) \}|
> |E(C)| - (|E(C)|-2)=2$,
and so there exist two distinct edges $e_1 = u_1v_1$ and $e_2 = u_2v_2$ on $C$ 
such that $T_{C,e_1}\setminus V(U_{C^*,e^*}) \neq \emptyset$ and 
$T_{C,e_2} \setminus  V(U_{C^*,e^*}) \neq \emptyset$. 
Take vertices 
$w_1 \in T_{C,e_1}\setminus V(U_{C^*,e^*})$ and 
$w_2 \in T_{C,e_2} \setminus  V(U_{C^*,e^*})$. 
Then $u_1 w_1 v_1$ and $u_2 w_2 v_2$
are $C$-avoiding paths.
Since $u_1$, $v_1$, $u_2$, $v_2$ are on the hole $C$, 
at least one pair of vertices in $T:=\{u_1, v_1, u_2, v_2\}$ is not adjacent. 
Therefore 
there exists a vertex in $T$ but not in $X_{C^*,e^*}$.
Without loss of generality,
we may assume that $u_1 \not\in X_{C^*,e^*}$.
By Claim \ref{Clm-1},
$T_{C,e_1} \subset V(U_{C^*,e^*}) \cup X_{C^*,e^*}$.
Since $T_{C,e_1} \cap V(U_{C^*,e^*}) = \emptyset$,
$T_{C,e_1} \subset X_{C^*,e^*}$ and so $w_1 \in X_{C^*,e^*}$.
Suppose that $w_2 \in X_{C^*,e^*}$.
Since $X_{C^*,e^*}$ is a clique by Lemma~\ref{lem:1} (2),
$w_1$ and $w_2$ are adjacent.
Then there exist both a $C'$-avoiding $(u_1,u_2)$-path and a
$C$-avoiding $(u_1,v_2)$-path,
contradicting Lemma~\ref{lem:1} (1).
Thus $w_2 \not\in X_{C^*,e^*}$.
Then $w_2 \not\in V(U_{C^*,e^*}) \cup X_{C^*,e^*}$.
By Claim \ref{Clm-1},
$\{u_2,v_2\} \subset V(U_{C^*,e^*}) \cup X_{C^*,e^*}$.
This implies that $u_1 w_1 u_2$ and $u_1 w_1 v_2$
are $C$-avoiding paths, which contradicts Lemma~\ref{lem:1} (1).
\end{proof}

\begin{Clm}\label{Clm-3}
For some $e' \in \{e \in E(C) \mid T_{C,e} \subset V(U_{C^*,e^*}) \}$,
there is no $C$-avoiding path from
any vertex in $T_{C,e'}$ to any vertex in $X_{C^*,e^*}$ in $G$.
\end{Clm}

\begin{proof}[Proof of Claim \ref{Clm-3}]
By Claim \ref{Clm-2},
there exist two distinct edges $e_1 = u_1v_1$ and $e_2 = u_2v_2$
in $\{e \in E(C) \mid T_{C,e} \subset V(U_{C^*,e^*})\}$.
That is, $T_{C,e_1}\subset V(U_{C^*,e^*})$
and $T_{C,e_2} \subset V(U_{C^*,e^*})$.
Suppose that $G$ has $C$-avoiding paths $P_1$ and $P_2$
from $w_1$ to a vertex in $X_{C^*,e^*}$
and from $w_2$ to a vertex in $X_{C^*,e^*}$,
respectively, for some  $w_1 \in T_{C,e_1}$
and $w_2 \in T_{C,e_2}$.
Then $P_1 P_2^{-1}$ contains a $C$-avoiding $(w_1,w_2)$-path.
However, this path extends to a $C$-avoiding
$(u_1,v_2)$-path, which contradicts Lemma~\ref{lem:1} (1).
This argument implies that for at
least one of $T_{C,e_1}$, $T_{C,e_2}$,
$G$ has no $C$-avoiding path from any of its vertices
to any vertex in $X_{C^*,e^*}$.
Without loss of generality, we may assume that $T_{C,e_1}$
satisfies this property
(for, otherwise, we can relabel the vertices on $C$
so that the vertex $u_2$ is labeled as $u_1$).
\end{proof}

\begin{Clm}\label{Clm-4}
For some $e' \in \{e \in E(C) \mid T_{C,e} \subset V(U_{C^*,e^*}) \}$,
\[
V(U_{C,e'}) \cup X_{C,e'} \subsetneq V(U_{C^*,e^*}) \cup X_{C^*,e^*}.
\]
\end{Clm}

\begin{proof}[Proof of Claim \ref{Clm-4}]
By Claim \ref{Clm-3},
there exists
$e_1=u_1v_1 \in \{e \in E(C) \mid T_{C,e} \subset V(U_{C^*,e^*}) \}$,
such that $G$ has no $C$-avoiding path from
any vertex in $T_{C,e}$ to any vertex in $X_{C^*,e^*}$.
Since $V(C) \subset V(U_{C^*,e^*}) \cup X_{C^*,e^*}$ by the choice of $C$, 
it holds that
$\{u_1,v_1\} \subset V(U_{C^*,e^*}) \cup X_{C^*,e^*}$.
Now take a vertex $x$ in $X_{C}$.
If $x \not\in X_{C^*,e^*}$, then $x$ is still adjacent to a vertex on $C$
in $G-X_{C^*,e^*}$ and so $x \in V(U_{C^*,e^*})$.
Therefore $X_{C} \subset V(U_{C^*,e^*}) \cup X_{C^*,e^*}$
and thus $X_{C,e_1} = X_{C} \cup \{ u_1,v_1 \} \subset
V(U_{C^*,e^*}) \cup X_{C^*,e^*}$.
Now it remains to show that
$V(U_{C,e_1}) \subset V(U_{C^*,e^*}) \cup X_{C^*,e^*}$.
Take a vertex $y$ in $U_{C,e_1}$.
Then $y$ belongs to a component $W$ of $G- X_{C,e_1}$.
By the definition of $U_{C,e_1}$,
$V(W) \cap T_{C,e_1} \neq \emptyset$.
Take a vertex $z$ in $V(W) \cap T_{C,e_1}$.
Then, since any vertex in $W$ and $z$ belong to a component of $U_{C,e_1}$,
any vertex in $W$ and $z$ are connected by a $C$-avoiding path.
Thus, by Claim \ref{Clm-3},
$W \cap X_{C^*,e^*} =\emptyset$
and so $W$ is a connected subgraph of $G - X_{C^*,e^*}$.
Since $T_{C,e_1} \subset V(U_{C^*,e^*})$,
we have $z \in V(U_{C^*,e^*})$.
Therefore, $V(W) \subset V(U_{C^*,e^*})$ since $z$ belongs to $W$,
which is connected in $G-X_{C^*,e^*}$.
Since $y \in V(W)$, we have $y \in V(U_{C^*,e^*})$.
We have just shown that $V(U_{C,e_1}) \subset V(U_{C^*,e^*})$.
Hence $V(U_{C,e_1}) \cup X_{C,e_1} \subset V(U_{C^*,e^*}) \cup X_{C^*,e^*}$.
Furthermore, by Claim \ref{Clm-2},
there is another edge $e_2=u_2 v_2 \in E(C)$ such that
$T_{C,e_2} \subset V(U_{C^*,e^*})$.
Now take $w_2$ in $T_{C,e_2}$.
Then $w_2 \in V(U_{C^*,e^*}) \cup X_{C^*,e^*}$.
However, $w_2 \not\in V(U_{C,e_1}) \cup X_{C,e_1}$
since $w_2$ is still adjacent to at least one of $u_2$,
$v_2$ in $G- X_{C,e_1}$.
Thus $V(U_{C,e_1}) \cup X_{C,e_1}
\subsetneq V(U_{C^*,e^*}) \cup X_{C^*,e^*}$
and the claim follows.
\end{proof}

To complete the proof, we denote by ${\mathcal H}_{C,e}$ 
the set of holes in $G[V(U_{C,e}) \cup X_{C,e} ]$ 
for $C \in \cH(G)$ and $e \in E(C)$. 
Let $C_1 \in \cH(G)$ and $e_1 \in E(C_1)$.
By our assumption that $G$ does not have the chordal property, 
there exists a hole $C_2 \in {\mathcal H}_{C_1,e_1}$.
By Claim \ref{Clm-4},
there exists
$e_2 \in \{e \in E(C_2) \mid T_{C_2,e} \subset V(U_{C_1,e_1})\}$
such that $V(U_{C_2,e_2}) \cup X_{C_2,e_2}
\subsetneq V(U_{C_1,e_1}) \cup X_{C_1,e_1}$.
Again, by our
assumption, 
there exists a hole $C_{3} \in \cH_{C_2,e_2}$.
Then, by Claim \ref{Clm-4}, there exists
$e_3 \in \{e \in E(C_3) \mid T_{C_3,e} \subset V(U_{C_2,e_2})\}$
such that $V(U_{C_3,e_3}) \cup X_{C_3,e_3} \subsetneq
V(U_{C_2,e_2}) \cup X_{C_2,e_2}$.
Repeating this process,
we have $C_1, C_2, \ldots, C_i, \ldots$
and $e_1, e_2, \ldots, e_i, \ldots$
such that
\[
V(U_{C_1,e_1}) \cup X_{C_1,e_1}
\supsetneq
V(U_{C_2,e_2}) \cup X_{C_2,e_2}
\supsetneq
\cdots
\supsetneq
V(U_{C_i,e_i}) \cup X_{C_i,e_i}
\supsetneq
\cdots,
\]
which is impossible since
$V(U_{C_1,e_1}) \cup X_{C_1,e_1}$ is
finite.
This completes the proof.
\end{proof}

The following theorem gives another sufficient condition for
the existence of a chordal cut.

\begin{Thm}\label{thm:1}
Let $G$ be a $K_{2,2,2}$-free hole-edge-disjoint graph.
Suppose that
there exists a hole $C \in \cH(G)$
such that
\[
|\{ e \in E(C) \mid S_{C,e}^{(G)} \neq \emptyset \}| \geq h(G).
\]
Then $G$ has the chordal property.
Consequently, $G$ has a chordal cut.
\end{Thm}

\begin{proof}
Let $C$ be a hole of $G$
such that
$|\{ e \in E(C) \mid S_{C,e} \neq \emptyset \}| \geq h(G)$.
Then there exists an edge $e \in E(C)$
such that $S_{C,e} \neq \emptyset$ and
$G[V(U_{C,e}) \cup X_{C,e}]$ does not contain
any holes, i.e., $G[V(U_{C,e}) \cup X_{C,e}]$ is a chordal graph. 
By Proposition \ref{prop:12},
$X_{C,e}$ is a clique cut.
Therefore $X_{C,e}$ is a chordal cut of $G$ 
and thus the theorem holds.
\end{proof}

\section{An Application}\label{sec:comp}

Kim~\cite{compone} conjectured that
$k(G) \leq h(G)+1$ holds for a graph $G$
(see \cite{ck, Kamibeppu, compone, KLS, KLPS, twoholes, maxclique,
E1, indep, twoholes2} for the studies on this conjecture).
It was shown in \cite{indep} (see also \cite{KLS}) that 
this conjecture is true for 
any $K_{2,2,2}$-free hole-edge-disjoint graph. 
Theorem \ref{thm:chordalcut} gives another proof for it.

\begin{Thm}\label{thm:appl}
If $G$ is a $K_{2,2,2}$-free hole-edge-disjoint graph with exactly
$h$ holes, then the competition number of $G$ is at most $h+1$.
\end{Thm}

\begin{proof}
We prove by induction on $h=h(G)$.
The case $h=0$ corresponds to a theorem by Roberts
\cite[Corollary 3]{cn}.
Suppose that the statement holds
for any $K_{2,2,2}$-free hole-edge-disjoint
graph with exactly $h-1$ holes for $h \geq 1$.
Let $G$ be a $K_{2,2,2}$-free hole-edge-disjoint graph
with exactly $h$ holes.

First, we consider the case where 
there exist $C \in \cH(G)$ and $e \in E(C)$
such that $T_{C,e} = \emptyset$. 
Then, by Proposition~\ref{ST}, $S_{C,e}=\emptyset$.
By Lemma \ref{lem:3},
$G-e$ is a $K_{2,2,2}$-free hole-edge-disjoint
graph with at most $h-1$ holes.
By induction hypothesis, there exists an acyclic digraph $D'$
such that $C(D')=(G-e) \cup I_h$, where
$I_h$ is the set of $h$ new vertices.
We define a digraph $D$ from $D'$ by
$V(D) = V(D') \cup \{ z \}$ and
$A(D) = A(D') \cup \{(u,z), (v,z)\}$,
where $uv=e$ and $z$ is a new vertex.
Then it is easy to
check that $D$ is acyclic and that
$C(D)=G \cup I_h \cup \{z\}$.

Next, we consider the case where
$T_{C,e} \neq \emptyset$ for any $C \in \cH(G)$ and any $e \in E(C)$. 
Then, by Proposition~\ref{ST}, $S_{C,e} \neq \emptyset$. 
Therefore, by Theorem~\ref{thm:chordalcut}, 
$G$ has the chordal property.
That is, there exist a hole $C^* \in \cH(G)$ and an edge $e^* \in E(C)$
such that $G[V(U_{C^*,e^*}) \cup X_{C^*,e^*} ]$ is chordal.
Let $H_{C^*,e^*}$ be the subgraph of $G$ induced by
$V(G) \setminus V(U_{C^*,e^*})$.
Then $H_{C^*,e^*}$ does not contain any $C$-avoiding
path by the definition of $U_{C^*,e^*}$.
Moreover, $H_{C^*,e^*}$ is $K_{2,2,2}$-free.
Thus the graph $G_1:=H_{C^*,e^*} - e^*$ contains at most $h-1$ holes
by Lemma \ref{lem:3}.
By the induction hypothesis, we have $k(G_1) \leq h$.
Let $G_2 := G[V(U_{C^*,e^*}) \cup X_{C^*,e^*} ]$.
Then $G_2$ is a chordal graph
and $X_{C^*,e^*}$ is a clique of $G_2$.
Moreover, $E(G_1) \cup E(G_2)=E(G)$, and
$V(G_1) \cap V(G_2)= X_{C^*,e^*}$.
Hence, by Theorem~\ref{chordalpart}, $k(G) \le h+1$.
\end{proof}



\begin{thebibliography}{99}

\bibitem{ck}
H. H. Cho and S. -R. Kim:
The competition number of a graph having exactly one hole,
{\it Discrete Mathematics} {\bf 303} (2005) 32--41.

\bibitem{co}
J. E. Cohen:
Interval graphs and food webs: a finding and a problem,
RAND Corporation Document 17696-PR, Santa Monica, CA, 1968.

\bibitem{Kamibeppu}
A. Kamibeppu:
An upper bound for the competition numbers of graphs,
{\it Discrete Applied Mathematics} {\bf 158} (2010) 154--157.

\bibitem{compone}
S. -R. Kim:
Graphs with one hole and competition number one,
{\it Journal of the Korean Mathematical Society} {\bf 42} (2005) 1251--1264.


\bibitem{KLS}
S. -R. Kim, J. Y. Lee, and Y. Sano:
The competition number of a graph whose holes do not overlap much,
{\it Discrete Applied Mathematics} {\bf 158} (2010) 1456--1460.

\bibitem{KLPS}
S. -R. Kim, J. Y. Lee, B. Park, and Y. Sano:
The competition number of a graph and the dimension of its hole space,
{\it Preprint}, arXiv:1103.1028


\bibitem{twoholes}
J. Y. Lee, S. -R. Kim, S. -J. Kim, and Y. Sano:
The competition number of a graph with exactly two holes,
{\it Ars Combinatoria} {\bf 95} (2010) 45--54.

\bibitem{maxclique}
J. Y. Lee, S. -R. Kim, S. -J. Kim, and Y. Sano:
Graphs having many holes but with small competition numbers,
{\it Applied Mathematics Letters},
doi:10.1016/j.aml.2011.03.003

\bibitem{E1}
J. Y. Lee, S. -R. Kim, and Y. Sano:
The competition number of a graph in which
any two holes share at most one edge,
{\it Preprint}, arXiv:1102.5718

\bibitem{indep}
B. -J. Li and G. J. Chang:
The competition number of a graph with exactly $h$ holes, all
of which are independent,
{\it Discrete Applied Mathematics} {\bf 157} (2009) 1337--1341.

\bibitem{twoholes2}
B. -J. Li and G. J. Chang:
The competition number of a graph with exactly two holes,
{\it Journal of Combinatorial Optimization},
DOI: 10.1007/s10878-010-9331-9


\bibitem{op}
{R. J. Opsut}:
{On the computation of the competition number of a graph},
{\it SIAM Journal on Algebraic and Discrete Methods} {\bf 3} (1982) 420--428.


\bibitem{cn}
{F. S. Roberts}:
{Food webs, competition graphs, and the boxicity of ecological phase space},
{\it Theory and applications of graphs (Proc. Internat. Conf.,
Western Mich. Univ., Kalamazoo, Mich., 1976)} (1978) 477--490.

\end{thebibliography}
\end{document}